\documentclass[pdflatex,sn-mathphys-num]{sn-jnl}

\usepackage{graphicx}%
\usepackage{multirow}%
\usepackage{amsmath,amssymb,amsfonts}%
\usepackage{amsthm}%
\usepackage{mathrsfs}%
\usepackage[title]{appendix}%
\usepackage{xcolor}%
\usepackage{textcomp}%
\usepackage{manyfoot}%
\usepackage{booktabs}%
\usepackage{algorithm}%
\usepackage{algorithmicx}%
\usepackage{algpseudocode}%
\usepackage{listings}

\DeclareMathOperator{\Tr}{Tr}  
\DeclareMathOperator{\SU}{SU}
\DeclareMathOperator{\U}{U}
\DeclareMathOperator{\SO}{SO}
\DeclareMathOperator{\diag}{diag}
\DeclareMathOperator{\sech}{sech}
\theoremstyle{thmstyleone}%
\newtheorem{theorem}{Theorem}
\newtheorem{proposition}[theorem]{Proposition}%
\newtheorem{lemma}[theorem]{Lemma}

\theoremstyle{thmstyletwo}%
\newtheorem{remark}{Remark}%

\theoremstyle{thmstylethree}%

\begin{document}

\title[Vacuum-Orbit Calibrations and Global Minimality]{Global Minimality of Matrix-Valued Heteroclinic Connections between Vacuum Orbits}

\author*[1]{\fnm{Sisi} \sur{Guan}}
\email{sisiguan98@bicmr.pku.edu.cn}

\author[1]{\fnm{Yu} \sur{Cheng}}
\email{chengyu0710@zju.edu.cn}

\affil*[1]{%
\orgdiv{School of Mathematical Sciences},
\orgname{Zhejiang University},
\orgaddress{%
\city{Hangzhou},
\postcode{310058},
\country{China}}}

\presentaddress{%
Sisi Guan, Beijing International Center for Mathematical Research, Peking University, Beijing 100871, China.}
\abstract{We study global energy minimization for matrix-valued heteroclinic connections whose vacuum sets are continuous conjugacy orbits rather than isolated points.  For two distinguished parameter regimes of a one-dimensional $\SU(5)\times\mathbb Z_2$ matrix field model, we first identify the full zero set of the quartic potential as the disjoint union of two compact homogeneous vacuum orbits.  We then construct auxiliary potentials controlled by the Frobenius distance to these orbits.  Pointwise comparison with the original potential, combined with a calibration inequality and the one-dimensional coarea formula, gives a sharp lower bound for every finite-energy connection joining the two vacuum components; no diagonality, commutativity, or symmetry-reduced ansatz is imposed on the competitors.  Two explicit non-Abelian kinks attain the bound and are therefore global minimizers in the full orbit-to-orbit heteroclinic classes. Their exact wall tensions are $4\sqrt2\,\mu^3/(3\lambda)$ and $9\sqrt2\,\mu^3/(10h)$, respectively.  Equivalently, the corresponding line-segment paths realize the degenerate weighted distance between the vacuum manifolds.  We also resolve every fixed-endpoint sector: for arbitrary prescribed representatives on the two vacuum orbits, the energy infimum is the same wall tension, and it is attained if and only if the prescribed representatives form a closest pair of the two orbits.  Nonclosest endpoint sectors therefore exhibit non-attainment through increasingly slow tangential motion along the vacuum manifolds. The proof separates a model-independent distance--coarea criterion from the model-specific orbit geometry and polynomial positivity estimates, providing a reusable mechanism for matrix-valued multiwell potentials with symmetry-generated vacua.}


\maketitle

\section{Introduction}
Nonconvex energies with multiple or symmetry-generated wells form a basic variational framework for phase transitions, interfacial layers, and topological defects \cite{Modica1987,Sternberg1988,BethuelBrezisHelein1994,SandierSerfaty2007}.  In the classical scalar theory, the potential has finitely many isolated minima, and a transition between two phases is described by an ordered one-dimensional profile. The work of Modica and the subsequent vectorial extensions identify the corresponding interfacial cost through a variational problem in the order-parameter space \cite{Modica1987,Sternberg1988,FonsecaTartar1989,Baldo1990}. For vector-valued energies, the transition cost depends on the geometry of admissible paths in the order-parameter space rather than on a single ordered scalar profile. It is therefore naturally connected with a weighted geodesic problem, and different routes through a multiwell landscape may carry different costs \cite{Baldo1990,AlikakosBeteluChen2006,AlikakosFusco2008,ZunigaSternberg2016,MonteilSantambrogio2018}.  These issues underlie both the sharp-interface theory of multiphase systems and the analysis of stationary transition layers.

At the profile level, the relevant objects are heteroclinic connections. For a nonnegative potential $W$ on a finite-dimensional Hilbert space, a heteroclinic is a finite-action trajectory joining two components of $W^{-1}(0)$ at the two ends of the real line. When the wells are isolated, existence and action minimization have been studied by direct variational methods, Hamiltonian techniques, and weighted-metric formulations \cite{Rabinowitz1993,Sternberg1991,AlikakosFusco2008,ZunigaSternberg2016,MonteilSantambrogio2018}.  The elementary inequality between the kinetic and potential terms yields a weighted-length lower bound, and equipartition is the corresponding equality condition.  This formulation also exhibits a central difficulty of genuine multiwell problems: a minimizing route between two prescribed wells may pass through another well or through a geometrically cheaper part of the zero set \cite{AlikakosBeteluChen2006,AlikakosFusco2008,ZunigaSternberg2016,MonteilSantambrogio2018}.

The present problem has an additional feature that is absent from the usual finite-well setting. Owing to the internal symmetry of the matrix order parameter, the vacuum set is not discrete: its components are positive-dimensional conjugacy orbits. Thus the transition problem is manifold-to-manifold rather than point-to-point. Tangential directions along a vacuum orbit are energetically degenerate, nearest vacuum representatives need not be unique, and an admissible path may leave any symmetry-reduced or diagonal ansatz while retaining the prescribed limits. Potentials with higher-dimensional wells occur in several phase-transition models and require geometric information beyond that used for isolated minima; see, for example, \cite{LPW}.  In the setting considered here, the decisive quantity is the Frobenius distance to the two vacuum orbits.  A main purpose of this paper is to show that this orbit geometry can be converted into a sharp global variational bound.

This structure arises in grand unified models with a $\mathbb Z_2$-invariant self-interaction. Non-Abelian kinks and their symmetry-breaking patterns have been studied for $\SO(10)$, $\SU(5)$, and $\operatorname{E}_6$ theories; see, for example, \cite{SV,PV,V,MNPS,PR,DGK}. Pogosian and Vachaspati constructed the $\SU(5)$ domain-wall solutions relevant here and analyzed their symmetry structure and local stability \cite{PV}; Vachaspati embedded one of these solutions in a wider $\SU(N)\times\mathbb Z_2$ family \cite{V}. A second $\SU(5)$ kink was later shown to be locally stable through the spectral analysis of a Schr\"odinger-type operator \cite{GRC}. Here local and global minimality address different variational questions. Local minimality controls competitors in a sufficiently small neighborhood of a given profile, whereas global minimality requires comparison with every admissible connection in the prescribed heteroclinic class. The distinction is essential in the present matrix-valued setting: spectral positivity does not exclude a lower-energy competitor that moves in non-diagonal directions, changes its asymptotic representatives along the vacuum orbits, or crosses the potential landscape through a different route.

We consider the static energy associated with the $(1+1)$-dimensional Lagrangian density
\begin{equation}\label{equ:density}
    \mathcal L=\Tr\bigl((\partial_z\Phi)^2\bigr)-V(\Phi).
\end{equation}
The field takes values in the real Hilbert space
\begin{equation}\label{equ:h0}
    \mathfrak h_0(5):=\{Q\in\mathbb C^{5\times5}:Q^\dagger=Q,\ \Tr Q=0\},
\end{equation}
equipped with the Frobenius inner product $\langle A,B\rangle_F:=\Tr(AB)$ and norm $\lVert A\rVert_F:=(\Tr(A^2))^{1/2}$.  The potential is
\begin{equation}\label{equ:potential}
    V(\Phi)=-\mu^2\Tr(\Phi^2)+h\bigl(\Tr(\Phi^2)\bigr)^2+\lambda\Tr(\Phi^4)+V_0,
\end{equation}
where $\mu>0$. The additive constant $V_0$ is chosen in each parameter regime so that $\min_{\mathfrak h_0(5)}V=0$. We prove below that, in each of the two regimes considered here, the full zero set consists exactly of two disjoint compact $\SU(5)$-conjugacy orbits. Thus the relevant transition is genuinely between two positive-dimensional vacuum components.

The two explicit solutions studied in this paper are
\begin{equation}\label{equ:phia}
    \Phi_A(z)=\frac{\sqrt{5}\mu}{2\sqrt{\lambda}}\left[\diag(1,1,-1,-1,0)+\frac{1}{5}\tanh\left(\frac{\mu z}{\sqrt{2}}\right)\diag(1,1,1,1,-4)\right]
\end{equation}
for $\lambda>0$ and $h=-3\lambda/20$, and
\begin{equation}\label{equ:phib}
    \Phi_B(z)=\frac{3\mu}{4\sqrt{h}}\left[\diag(0,0,0,1,-1)+\frac{1}{5}\tanh\left(\frac{\mu z}{\sqrt{2}}\right)\diag(2,2,2,-3,-3)\right]
\end{equation}
for $h>0$ and $\lambda=-10h/9$.  They induce, respectively,
\begin{equation}\label{equ:breaking}
    \SU(5)\times\mathbb Z_2\longrightarrow\frac{\SU(3)\times\SU(2)\times\U(1)}{\mathbb Z_3\times\mathbb Z_2},\qquad\SU(5)\times\mathbb Z_2\longrightarrow\frac{\SU(4)\times\U(1)}{\mathbb Z_4}.
\end{equation}
The special parameter relations are those for which the reduced Euler--Lagrange systems decouple and the profiles are explicit \cite{PV,V,GRC}.

The contribution of this paper is threefold. First, we identify the complete zero set of the quartic potential in the two distinguished parameter regimes and show that it consists of two disjoint compact homogeneous conjugacy orbits. This converts the wall problem from a point-to-point connection problem into an orbit-to-orbit one and makes the vacuum geometry part of the variational analysis.

Second, we establish a model-independent distance--coarea criterion. If two disjoint compact zero sets $S^-$ and $S^+$ satisfy
\[
    W(q)\ge g\bigl(\min\{d(q,S^-),d(q,S^+)\}\bigr),
\]
then every finite-energy connection between them satisfies a sharp lower bound depending only on $g$ and $d(S^-,S^+)$. A companion rigidity statement shows that equality forces the prescribed endpoint representatives to form a closest pair.  These results isolate the variational mechanism from the specific $\SU(5)$ algebra.

Third, for each parameter regime we construct a distance-controlled auxiliary potential and prove the pointwise comparison $V\ge V_A$ or $V\ge V_B$. The comparison is the model-specific part of the proof and requires optimal spectral matching together with exact polynomial positivity estimates. The explicit profiles $\Phi_A$ and $\Phi_B$ traverse minimizing line segments, satisfy equipartition, and saturate every step of the distance--coarea bound. Consequently, they are global minimizers among all finite-energy matrix-valued connections whose limits lie on the corresponding vacuum orbits. Neither the path nor its asymptotic representatives are restricted to a diagonal, commuting, or symmetry-reduced class, and the exact wall tensions are obtained at the same time.

The continuous vacuum geometry also produces a fixed-endpoint phenomenon that has no analogue for isolated wells. For every prescribed pair $A^-\in\Sigma_*^-$ and $A^+\in\Sigma_*^+$, the energy infimum in $\mathcal A(A^-,A^+)$ equals the same orbit-to-orbit wall tension. The infimum is attained precisely when $(A^-,A^+)$ realizes the distance between the two vacuum orbits. For a nonclosest pair, one can approach the optimal energy by moving increasingly slowly along each zero-energy orbit before and after the central kink, but no minimizer exists. Thus the positive-dimensional vacuum set leads to an explicit loss-of-compactness mechanism in fixed-endpoint heteroclinic sectors.

The resulting proof separates the geometry of the vacuum set from the algebra of the particular potential. It therefore provides more than a stability calculation for two explicit profiles: it gives a reusable route to global minimality in matrix-valued multiwell energies with continuous symmetry-generated vacua. At the same time, the pointwise orbit-distance comparison is genuinely restrictive and records where the model-specific difficulty lies. Extending the result to general coupling ratios, different pairs of vacuum orbits, or larger $\SU(N)$ models amounts to establishing new sharp comparison inequalities rather than modifying the variational criterion itself.

The paper is organized as follows.  Section~\ref{sec:framework} sets up the variational framework, characterizes the complete vacuum sets and their homogeneous-space geometry, computes the orbit separation, and states the main theorems. Section~\ref{sec:auxiliary} constructs the auxiliary potentials and proves the pointwise orbit-distance estimates. Section~\ref{sec:criterion} establishes the general distance--coarea criterion and a rigidity statement for sharp equality. Sections~\ref{sec:app-a} and~\ref{sec:app-b} verify the equality conditions for $\Phi_A$ and $\Phi_B$, respectively, and identify the exact weighted transition costs. Section~\ref{sec:fixed-endpoints} analyzes arbitrary fixed-endpoint sectors, including the attainment criterion and the non-attainment mechanism for nonclosest endpoint pairs. Section~\ref{sec:conclusion} concludes, and the appendix records exact positivity certificates used in the model-specific estimates.

\section{Variational setting and vacuum manifolds}\label{sec:framework}
For $A^-,A^+\in\mathfrak h_0(5)$, define the fixed-endpoint class
\begin{equation*}
\begin{aligned}
    \mathcal A(A^-,A^+):=\{&\Psi\in H^1_{\mathrm{loc}}(\mathbb R;\mathfrak h_0(5)):\,\Psi'\in L^2(\mathbb R;\mathfrak h_0(5)),\\
    &V(\Psi)\in L^1(\mathbb R),\lim_{z\to-\infty}\Psi(z)=A^-,\lim_{z\to+\infty}\Psi(z)=A^+\}.
\end{aligned}
\end{equation*}
For two nonempty compact sets $S^-,S^+\subset\mathfrak h_0(5)$, we also use the orbit-to-orbit class
\begin{equation*}
    \mathcal A(S^-,S^+):=\bigcup_{A^-\in S^-,\,A^+\in S^+}\mathcal A(A^-,A^+).
\end{equation*}
A minimizer of $E$ in $\mathcal A(A^-,A^+)$ will be called a fixed-endpoint global minimizer, while a minimizer in $\mathcal A(S^-,S^+)$ will be called an orbit-to-orbit global minimizer. The latter class allows both endpoint representatives to vary on the two vacuum components.

For every admissible map, set
\begin{equation}\label{equ:energy}
    E(\Psi):=\int_{\mathbb R}\left(\lVert\Psi'(z)\rVert_F^2+V(\Psi(z))\right)\,\mathrm dz.
\end{equation}
The use of $H^1_{\mathrm{loc}}$, rather than $H^1(\mathbb R)$, is necessary because a heteroclinic profile with nonzero limits need not be square integrable.

For the first parameter regime let
\begin{equation}\label{equ:parameters-a}
    \lambda>0,\qquad h=-\frac{3}{20}\lambda,\qquad u:=\frac{\sqrt{5}\mu}{5\sqrt{\lambda}},\qquad V_0:=\frac{3\mu^4}{\lambda}.
\end{equation}
For the second let
\begin{equation}\label{equ:parameters-b}
    h>0,\qquad \lambda=-\frac{10}{9}h,\qquad v:=\frac{3\mu}{10\sqrt h},\qquad V_0:=\frac{9\mu^4}{10h}.
\end{equation}
The endpoint matrices are
\begin{align}
    \Phi_A^+&=u\diag(3,3,-2,-2,-2), & \Phi_A^-&=u\diag(2,2,-3,-3,2),\label{equ:endpoints-a}\\ 
    \Phi_B^+&=v\diag(1,1,1,1,-4), & \Phi_B^-&=v\diag(-1,-1,-1,4,-1).\label{equ:endpoints-b}
\end{align}
Their $\SU(5)$ conjugacy orbits are
\begin{equation*}
    \Sigma_\ast^\pm:= \{U^\dagger\Phi_\ast^\pm U:U\in\SU(5)\},\qquad \ast\in\{A,B\}.
\end{equation*}

We begin with the trace inequality that determines both the vacuum set and the normalization of the potential.

\begin{lemma}[Fourth-moment bounds]\label{lem:fourth-moment}
Let $x_1,\ldots,x_5\in\mathbb R$ satisfy $\sum_{i=1}^5x_i=0$, and put $p_j:=\sum_{i=1}^5x_i^j$.  Then
\begin{equation*}
    \frac{7}{30}p_2^2\le p_4\le \frac{13}{20}p_2^2.
\end{equation*}
The lower equality occurs, up to scaling and permutation, at $(3,3,-2,-2,-2)$, and the upper equality occurs, up to scaling and permutation, at $(4,-1,-1,-1,-1)$.
\end{lemma}

\begin{proof}
By homogeneity it is enough to impose $p_2=1$.  At an extremum of $p_4$ under the constraints $p_1=0$ and $p_2=1$, the Lagrange multiplier equations have the form
\[
    4x_i^3-2\sigma_2x_i-\sigma_1=0,\qquad i=1,\ldots,5.
\]
Thus the coordinates take at most three distinct values. If there are two distinct values with multiplicities $m$ and $5-m$, the constraints give
\[
    p_4=\frac1{25}\left(\frac{(5-m)^2}{m}+\frac{m^2}{5-m}\right).
\]
For $m=1,4$ this is $13/20$, and for $m=2,3$ it is $7/30$. If three distinct roots $a,b,c$ occur, then $a+b+c=0$ because the cubic above has no quadratic term. The possible multiplicity partitions are $(3,1,1)$ and $(2,2,1)$, up to permutation. Combining the weighted trace constraint with $a+b+c=0$ shows that one value is zero and the other two are opposite; hence $p_4=1/2$.  This lies strictly between the two asserted extrema.
\end{proof}

The sharp equality cases in Lemma~\ref{lem:fourth-moment} identify the only eigenvalue multiplicity patterns that can annihilate the quartic potential. They therefore determine the full vacuum set, rather than merely providing a lower bound for $V$.

\begin{proposition}[Exact vacuum characterization]\label{prop:vacuum-characterization}
In the parameter regime \eqref{equ:parameters-a},
\[
    V\ge0,\qquad V^{-1}(0)=\Sigma_A^-\cup\Sigma_A^+.
\]
In the parameter regime \eqref{equ:parameters-b},
\[
    V\ge0,\qquad V^{-1}(0)=\Sigma_B^-\cup\Sigma_B^+.
\]
In both cases the two components are compact, connected, and disjoint.
\end{proposition}

\begin{proof}
For the $A$-type regime, diagonalize $Q=u\,U^\dagger\diag(x_1,\ldots,x_5)U$, where $\sum_i x_i=0$. Direct substitution gives
\begin{equation*}
    \frac{\lambda}{\mu^4}V(Q)=\frac{(p_2-30)^2}{300}+\frac1{25}\left(p_4-\frac7{30}p_2^2\right)\ge0
\end{equation*}
by Lemma~\ref{lem:fourth-moment}. Equality requires $p_2=30$ and equality in the lower fourth-moment bound.  Its equality characterization implies that the eigenvalue vector is, up to permutation, $\pm(3,3,-2,-2,-2)$; these two signs give precisely $\Sigma_A^+$ and $\Sigma_A^-$.

For the $B$-type regime, write $Q=v\,U^\dagger\diag(x_1,\ldots,x_5)U$.  Then
\begin{equation*}
    \frac{10h}{9\mu^4}V(Q)=\frac{(p_2-20)^2}{400}+\frac1{100}\left(\frac{13}{20}p_2^2-p_4\right)\ge0.
\end{equation*}
Equality requires $p_2=20$ and equality in the upper fourth-moment bound, so the eigenvalue vector is, up to permutation, $\pm(4,-1,-1,-1,-1)$.  These are exactly the two orbits $\Sigma_B^\pm$. Compactness and connectedness follow from compactness and connectedness of $\SU(5)$; disjointness follows from the distinct spectra.
\end{proof}

The preceding proposition determines the vacuum components spectrally. Their intrinsic geometry records the zero-energy directions available to an admissible connection and will be used later in the fixed-endpoint recovery construction.

\begin{remark}[Homogeneous-space geometry]\label{rmk:orbit-geometry}
The stabilizer of an $A$-type vacuum is $S(\U(2)\times\U(3))$, whereas that of a $B$-type vacuum is $S(\U(4)\times\U(1))$.  Consequently,
\[
    \Sigma_A^\pm\simeq \frac{\SU(5)}{S(\U(2)\times\U(3))}\simeq \operatorname{Gr}_{\mathbb C}(2,5),\qquad \dim_{\mathbb R}\Sigma_A^\pm=12,
\]
and
\[
    \Sigma_B^\pm\simeq\frac{\SU(5)}{S(\U(4)\times\U(1))}\simeq \mathbb{CP}^4,\qquad \dim_{\mathbb R}\Sigma_B^\pm=8.
\]
Thus the zero-energy tangential directions are precisely the tangent directions of these compact homogeneous manifolds.
\end{remark}

To formulate the transition cost geometrically, we next compute the separation of the two vacuum components in the ambient Frobenius metric. For nonempty compact sets $S,T\subset\mathfrak h_0(5)$ define
\begin{equation*}
    d_F(Q,S):=\min_{P\in S}\lVert Q-P\rVert_F,\qquad d_F(S,T):=\min_{\substack{P\in S\\Q\in T}}\lVert P-Q\rVert_F.
\end{equation*}

\begin{lemma}[Distance between unitary orbits]\label{lem:orbit-distance}
Let $A,B$ be Hermitian matrices with eigenvalues $a_1\ge\cdots\ge a_5$ and $b_1\ge\cdots\ge b_5$. Then
\begin{equation}\label{equ:unitary-orbit-distance}
    \min_{U\in\SU(5)}\lVert A-U^\dagger BU\rVert_F^2=\sum_{i=1}^5(a_i-b_i)^2.
\end{equation}
Consequently,
\begin{equation}\label{equ:orbit-distances}
    d_A:=d_F(\Sigma_A^-,\Sigma_A^+)=2\sqrt5\,u,\qquad d_B:=d_F(\Sigma_B^-,\Sigma_B^+)=\sqrt{30}\,v,
\end{equation}
and the pairs in \eqref{equ:endpoints-a}--\eqref{equ:endpoints-b} attain these distances.
\end{lemma}

\begin{proof}
Expanding the Frobenius norm reduces the minimization to maximizing $\Tr(AU^\dagger BU)$. The von Neumann trace inequality, followed by the rearrangement inequality, gives (see, e.g.,
\cite{HoffmanWielandt1953,Bhatia1997})
\[
    \max_{U\in\U(5)}\Tr(AU^\dagger BU)=\sum_{i=1}^5a_ib_i.
\]
Multiplying a maximizing unitary by a scalar phase leaves the conjugation unchanged, so the same maximum is attained in $\SU(5)$. This proves \eqref{equ:unitary-orbit-distance}. Substitution of the ordered spectra in \eqref{equ:endpoints-a} and \eqref{equ:endpoints-b} gives \eqref{equ:orbit-distances}.
\end{proof}

The next elementary observation describes the distance-to-vacuum function along any segment joining a closest pair. It is the geometric equality case that will allow the explicit kinks to saturate the coarea lower bound.

\begin{lemma}[Distance along a minimizing segment]\label{lem:minimizing-segment}
Let $S^-,S^+$ be nonempty compact subsets of a Hilbert space, let $D=d(S^-,S^+)$, and suppose $p^\pm\in S^\pm$ satisfy $\lVert p^+-p^-\rVert=D$. For $q_s=(1-s)p^-+sp^+$, $0\le s\le1$, one has
\[
    d(q_s,S^-)=sD,\qquad d(q_s,S^+)=(1-s)D.
\]
In particular, $\min\{d(q_s,S^-),d(q_s,S^+)\}=D\min\{s,1-s\}$.
\end{lemma}

\begin{proof}
The upper bounds follow by comparison with $p^-$ and $p^+$. For every $a\in S^-$,
\[
    \lVert q_s-a\rVert\ge \lVert p^+-a\rVert-\lVert p^+-q_s\rVert\ge D-(1-s)D=sD.
\]
Taking the infimum over $a\in S^-$ gives the first equality; the second is symmetric.
\end{proof}

For later comparison with the metric formulation of the connection problem, define the degenerate weighted distance
\begin{equation*}
    \mathsf d_V(S^-,S^+):=\inf_{\substack{\gamma\in AC([0,1];\mathfrak h_0(5))\\\gamma(0)\in S^-,\,\gamma(1)\in S^+}}2\int_0^1\sqrt{V(\gamma(t))}\,\lVert\gamma'(t)\rVert_F\,\mathrm dt.
\end{equation*}

The main results can now be stated in their natural orbit-to-orbit form. The first two theorems assert global minimality in the largest heteroclinic classes considered here; the third theorem explains what changes when the endpoint representatives are fixed in advance.

\begin{theorem}[Orbit-to-orbit global minimality: A-type kink]\label{thm:globala}
Under \eqref{equ:parameters-a}, the kink $\Phi_A$ defined in \eqref{equ:phia} is a global minimizer of $E$ in $\mathcal A(\Sigma_A^-,\Sigma_A^+)$.  More precisely,
\[
    E(\Psi)\ge E(\Phi_A)=\frac{4\sqrt2\,\mu^3}{3\lambda}\qquad\text{for every }\Psi\in\mathcal A(\Sigma_A^-,\Sigma_A^+).
\]
Moreover,
\[
    \mathsf d_V(\Sigma_A^-,\Sigma_A^+)=E(\Phi_A).
\]
\end{theorem}

\begin{theorem}[Orbit-to-orbit global minimality: B-type kink]\label{thm:globalb}
Under \eqref{equ:parameters-b}, the kink $\Phi_B$ defined in \eqref{equ:phib} is a global minimizer of $E$ in $\mathcal A(\Sigma_B^-,\Sigma_B^+)$.  More precisely,
\[
    E(\Psi)\ge E(\Phi_B)=\frac{9\sqrt2\,\mu^3}{10h}\qquad\text{for every }\Psi\in\mathcal A(\Sigma_B^-,\Sigma_B^+).
\]
Moreover,
\[
    \mathsf d_V(\Sigma_B^-,\Sigma_B^+)=E(\Phi_B).
\]
\end{theorem}

\begin{theorem}[Fixed-endpoint sectors and attainment]\label{thm:fixed-endpoint-sectors}
Set
\[
    \sigma_A:=\frac{4\sqrt2\,\mu^3}{3\lambda},\qquad\sigma_B:=\frac{9\sqrt2\,\mu^3}{10h}.
\]
For $\ast\in\{A,B\}$ and arbitrary $A^-\in\Sigma_\ast^-$, $A^+\in\Sigma_\ast^+$,
\begin{equation}\label{equ:fixed-endpoint-infimum}
    \inf_{\Psi\in\mathcal A(A^-,A^+)}E(\Psi)=\sigma_\ast.
\end{equation}
Moreover, the infimum in \eqref{equ:fixed-endpoint-infimum} is attained if and only if
\begin{equation}\label{equ:closest-endpoint-condition}
    \lVert A^+-A^-\rVert_F=d_\ast.
\end{equation}
When \eqref{equ:closest-endpoint-condition} holds, the pair $(A^-,A^+)$ is a simultaneous $\SU(5)$ conjugate of $(\Phi_\ast^-,\Phi_\ast^+)$, and a translated conjugate of the explicit kink is a minimizer. If $\lVert A^+-A^-\rVert_F>d_\ast$, the infimum is not attained.
\end{theorem}
\begin{remark}[Interpretation of the fixed-endpoint result]
\label{rmk:fixed-endpoint-interpretation}
Theorem~\ref{thm:fixed-endpoint-sectors} separates two features of the heteroclinic problem that coincide when the wells are isolated but become distinct for positive-dimensional vacuum manifolds. First, the optimal energy depends only on the two vacuum components and not on the prescribed representatives: every fixed-endpoint sector has the same infimum $\sigma_\ast$ as the corresponding orbit-to-orbit problem. Indeed, an admissible sequence may move increasingly slowly along the zero set before and after performing the optimal transition between a closest pair of vacuum representatives. Since the potential vanishes on the vacuum orbits, the energetic cost of these tangential motions can be made arbitrarily small.

Second, attainment retains geometric information about the prescribed endpoints. A minimizing connection exists precisely when the two representatives form a closest pair of the vacuum orbits. For nonclosest endpoints, the tangential motion required to reach such a pair is pushed farther and farther toward spatial infinity along a minimizing sequence, so compactness is lost and the infimum is not attained. This phenomenon is specific to continuous vacuum manifolds and has no direct analogue in the usual point-to-point problem with isolated wells. It also illustrates why global minimization in the full matrix-valued admissible class contains information that is not detected by local or spectral stability of an individual kink.
\end{remark}

\section{Construction of the auxiliary potentials}\label{sec:auxiliary}

We now reduce the matrix-valued potential to functions of the orbit distances introduced in Section~\ref{sec:framework}. We write $d(Q,\Sigma):=d_F(Q,\Sigma)$ throughout this section.  The auxiliary potentials used below are deliberately truncated to zero once the distance from both vacuum orbits reaches one half of the orbit separation. This removes an unnecessary polynomial estimate in the middle region; the coarea lower bound only uses levels strictly below $d_*/2$.

We first treat the $A$-type kink.  The following lemma is the main model-specific estimate for this regime: among real diagonal matrices at a prescribed distance from the vacuum orbit, the potential is bounded below by its value on the segment joining a closest vacuum pair.

\begin{lemma}[Diagonal A-type comparison]\label{lem:reala}
Let $\Psi\in\mathfrak h_0(5)$ be real and diagonal.  If
\[
    d(\Psi,\Sigma_{A,\mathbb R}^{+})=\delta_A\in[0,d_A/2],\qquad\Sigma_{A,\mathbb R}^{+}:=\{P^T\Phi_A^{+}P:P\in\SO(5)\},
\]
then
\[
    V(\Psi)\ge V\left(\frac{\delta_A}{d_A}\Phi_A^{-}+\left(1-\frac{\delta_A}{d_A}\right)\Phi_A^{+}\right).
\]
\end{lemma}

\begin{proof}
Write
\[
    \alpha:=\frac{\delta_A}{d_A}\in[0,1/2],\qquad\Lambda_A^+:=\alpha\Phi_A^-+(1-\alpha)\Phi_A^+,
\]
and, after permuting the diagonal entries, assume that
$\Psi_1\ge\Psi_2\ge\cdots\ge\Psi_5$ in
$\Psi=u\diag(\Psi_1,\ldots,\Psi_5)$.  By the Ky Fan maximum
principle, the closest point of $\Sigma_{A,\mathbb R}^{+}$ is obtained by
matching the two entries $3u$ with the two largest eigenvalues of $\Psi$.
Hence
\[
    \delta_A^2
    =u^2\left[(\Psi_1-3)^2+(\Psi_2-3)^2
    +\sum_{k=3}^5(\Psi_k+2)^2\right].
\]
Consequently, with
\[
    \widehat\Psi_1=\Psi_1-3,\quad
    \widehat\Psi_2=\Psi_2-3,\quad
    \widehat\Psi_k=\Psi_k+2\quad(k=3,4,5),
\]
we have
\begin{equation*}
    \sum_{k=1}^5\widehat\Psi_k=0,\qquad\sum_{k=1}^5\widehat\Psi_k^2=\frac{\delta_A^2}{u^2}=\frac{\alpha^2d_A^2}{u^2}=20\alpha^2.
\end{equation*}
A direct expansion yields
\begin{align}\label{equ:A-difference-exact}
    V(\Psi)-V(\Lambda_A^+)=\frac{\mu^4}{25\lambda}\bigg[&\sum_{k=1}^5\widehat\Psi_k^4-8\sum_{k=1}^5\widehat\Psi_k^3+20(\widehat\Psi_1^3+\widehat\Psi_2^3)+15(\widehat\Psi_1-\widehat\Psi_2)^2\notag\\
    &-60\alpha^2(\widehat\Psi_1+\widehat\Psi_2)-260\alpha^4+400\alpha^3\bigg].
\end{align}
The case $\alpha=0$ is immediate.  Assume $\alpha>0$ and set $z_k:=\widehat\Psi_k/\alpha$.  Then
\begin{equation}\label{equ:A-z-sphere}
    \sum_{k=1}^5z_k=0,\qquad \sum_{k=1}^5z_k^2=20.
\end{equation}
Define
\begin{align*}
    A_A(z)&:=\sum_{k=1}^5z_k^4-260,\\
    B_A(z)&:=-8\sum_{k=1}^5z_k^3+20(z_1^3+z_2^3)+30(z_1-z_2)^2-60(z_1+z_2)+400,\\
    \mathcal H_\alpha(z)&:=A_A(z)+\frac1\alpha B_A(z).
\end{align*}
Equation~\eqref{equ:A-difference-exact} becomes
\begin{equation}\label{equ:A-difference-H}
    V(\Psi)-V(\Lambda_A^+)=\frac{\mu^4\alpha^4}{25\lambda}\left[\mathcal H_\alpha(z)+\frac{15(1-2\alpha)}{\alpha^2}(z_1-z_2)^2\right].
\end{equation}
Since $0<\alpha\le 1/2$, we have $\alpha^{-1}\ge 2$. Hence, provided that $B_A(z)\ge0$,
\[
	\mathcal H_\alpha(z)=A_A(z)+\frac1\alpha B_A(z)\ge A_A(z)+2B_A(z)=\mathcal H_{1/2}(z).
\]
Moreover, the last term in \eqref{equ:A-difference-H} is nonnegative. It is therefore sufficient to prove
\[
	B_A(z)\ge0 \qquad\text{and}\qquad \mathcal H_{1/2}(z)\ge0
\]
on the constraint set \eqref{equ:A-z-sphere}.

\smallskip
\noindent\emph{Step 1: positivity of $B_A$.}
The constraint set in \eqref{equ:A-z-sphere} is compact, and the two constraint gradients are independent.  At a stationary point of $B_A$, there are multipliers $\sigma_1,\sigma_2$ such that
\begin{equation}\label{equ:BA-stationary}
\left\{
\begin{aligned}
    -24z_k^2+2\sigma_2z_k+\sigma_1&=0,&&k=3,4,5,\\
    36z_1^2+60(z_1-z_2)+2\sigma_2z_1+\sigma_1-60&=0,\\
    36z_2^2+60(z_2-z_1)+2\sigma_2z_2+\sigma_1-60&=0.
\end{aligned}
\right.
\end{equation}
Thus $z_3,z_4,z_5$ take at most two distinct values, while either $z_1=z_2$ or $z_1\ne z_2$. Up to permutations inside the two groups, the stationary points are exhausted by the following four cases.

If $z_3=z_4=z_5=\theta$ and $z_1=z_2=r$, the constraints give
\[
    (\theta,r)=\left(-\frac{2\sqrt6}{3},\sqrt6\right),\quad\left(\frac{2\sqrt6}{3},-\sqrt6\right),
\]
and the corresponding values of $B_A$ are $200(6+\sqrt6)/3$ and $200(6-\sqrt6)/3$.

If $z_3=z_4=z_5=\theta$ and $z_1\ne z_2$, put $s=z_1+z_2$ and $p=z_1z_2$. The constraints give $s=-3\theta$ and $p=6\theta^2-10$. The remaining equation in \eqref{equ:BA-stationary} reduces exactly to
\[
    \theta^2-\theta-1=0.
\]
Moreover $-2$ is one root of $X^2+3\theta X+6\theta^2-10$, so
\[
    \{z_1,z_2\}=\{-2,2-3\theta\},\qquad\theta=\frac{1\pm\sqrt5}{2}.
\]
The two values of $B_A$ are
\[
    25(43+15\sqrt5),\qquad 25(43-15\sqrt5)>0.
\]

If $z_3=z_4=a$, $z_5=b$ with $a
e b$, and $z_1=z_2=r$, the complete solution set is
\[
    (a,b,r)=(-1,4,-1),\ (1,-4,1),\ (\sqrt5,0,-\sqrt5),\ (-\sqrt5,0,\sqrt5),
\]
with corresponding values
\[
    0,\quad 800,\quad 80(5-\sqrt5),\quad 80(5+\sqrt5).
\]

Finally, suppose $z_3=z_4=a$, $z_5=b$ with $a
e b$, and $z_1,z_2$ are distinct. Writing $s=z_1+z_2$ and $p=z_1z_2$, the stationary equations and the two constraints imply
\[
    b=10-4a,\qquad s=2a-10,\qquad p=\frac{8a^2-30a+45}{3}, \qquad \frac{50}{3}(a-3)^2=0.
\]
Thus $a=3$, $b=-2$, $s=-4$, and $p=9$, but the quadratic $X^2-sX+p$ has discriminant $s^2-4p=-20$. Hence this case has no real stationary point. We conclude that
\begin{equation}\label{equ:BA-positive}
    B_A(z)\ge0\quad\text{whenever}\quad\sum z_k=0,\quad\sum z_k^2=20.
\end{equation}
In particular, $\mathcal H_\alpha(z)\ge\mathcal H_{1/2}(z)$ for $0<\alpha\le1/2$.

\smallskip
\noindent\emph{Step 2: positivity at $\alpha=1/2$.}
Set $w=z/2$, so that $\sum w_k=0$ and $\sum w_k^2=5$. A direct calculation gives
\[
    \mathcal H_{1/2}(z)=16\mathcal K(w),
\]
where
\begin{align*}
    \mathcal K(w):={}&\sum_{k=1}^5w_k^4-8\sum_{k=1}^5w_k^3+20(w_1^3+w_2^3)+15(w_1-w_2)^2\\
    &-15(w_1+w_2)+\frac{135}{4}.
\end{align*}
At a stationary point of $\mathcal K$ on this constraint sphere, there are
multipliers $\sigma_1,\sigma_2$ such that
\begin{equation}\label{equ:K-stationary}
\left\{
\begin{aligned}
4w_k^3-24w_k^2+2\sigma_2w_k+\sigma_1&=0,
&&k=3,4,5,\\
4w_1^3+36w_1^2+30(w_1-w_2)
 +2\sigma_2w_1+\sigma_1-15&=0,\\
4w_2^3+36w_2^2+30(w_2-w_1)
 +2\sigma_2w_2+\sigma_1-15&=0.
\end{aligned}
\right.
\end{equation}
Thus $w_3,w_4,w_5$ are roots of one cubic.  Subtracting the last two
equations gives
\[
(w_1-w_2)\left[4(w_1^2+w_1w_2+w_2^2)
+36(w_1+w_2)+60+2\sigma_2\right]=0.
\]
The resulting finite classification is detailed in
Appendix~\ref{app:a-stationary}.

If $w_3=w_4=w_5=\theta$ and $w_1=w_2=r$, then
\[
    (\theta,r)=\left(-\sqrt{\frac23},\sqrt{\frac32}\right),\quad\left(\sqrt{\frac23},-\sqrt{\frac32}\right),
\]
and
\[
    \mathcal K=\frac{25}{12}(19+4\sqrt6),\qquad\mathcal K=\frac{25}{12}(19-4\sqrt6).
\]

If $w_3=w_4=w_5=\theta$ and $w_1\ne w_2$, elimination of $w_1+w_2$ and $w_1w_2$ gives
\begin{equation}\label{equ:K-case2-cubic}
    q(\theta):=16\theta^3-60\theta^2+26\theta+15=0,\qquad 10-15\theta^2\ge0.
\end{equation}
On $[-\sqrt{2/3},0]$ we have $q'>0$, while $q(-1/3)<0<q(-3/10)$.  On $[0,\sqrt{2/3}]$, the only critical point is a local maximum, and
\[
    q(0)=15,\qquad q(\sqrt{2/3})=-25+\frac{110\sqrt6}{9}>0.
\]
Hence \eqref{equ:K-case2-cubic} has exactly one admissible root $\theta_0\in(-1/3,-3/10)$. The two special coordinates are the roots of
\[
    X^2+3\theta_0X+6\theta_0^2-\frac52=0.
\]
Reduction modulo $q(\theta_0)=0$ gives the exact value
\[
    \mathcal K=\frac{735}{4}\theta_0^2-\frac{2325}{8}\theta_0+\frac{2015}{16}>0.
\]
The positivity is immediate from $\theta_0<0$. This replaces the decimal
stationary points by an exact elimination and root-isolation argument.

If $w_3=w_4=a$, $w_5=b$ with $a
e b$, and $w_1=w_2=r$, the complete solution set is
\[
    (a,b,r)=\left(-\frac12,2,-\frac12\right),\left(\frac12,-2,\frac12\right),\left(\frac{\sqrt5}{2},0,-\frac{\sqrt5}{2}\right),\left(-\frac{\sqrt5}{2},0,\frac{\sqrt5}{2}\right),
\]
with corresponding values
\[
    0,\quad100,\quad10(4-\sqrt5),\quad10(4+\sqrt5).
\]

If $w_3=w_4=a$, $w_5=b$ with $a
e b$, and $w_1,w_2$ are distinct, the stationary equations imply
\begin{equation*}
    2ab-2b^2-24a-6b+35=0.
\end{equation*}
Since $|b|\le\sqrt5<12$, this determines $a=(2b^2+6b-35)/(2b-24)$. If $s=w_1+w_2$ and $p=w_1w_2$, the discriminant of the quadratic with roots $w_1,w_2$ then equals
\[
    s^2-4p=-\frac{5(3b^4-12b^3+48b+202)}{(b-12)^2}.
\]
For $|b|\le\sqrt5$,
\[
    3b^4-12b^3+48b+202=3(b^2-2b)^2-12(b-2)^2+250\ge142-48\sqrt5>0.
\]
Thus this pattern has no real stationary point. If $w_3,w_4,w_5$ are all distinct, they are the three roots of $4X^3-24X^2+2\sigma_2X+\sigma_1$, and hence their sum is $6$. Cauchy's inequality would give $w_3^2+w_4^2+w_5^2\ge12>5$, which is impossible.

All stationary values of $\mathcal K$ are therefore nonnegative. Since the constraint set is compact, $\mathcal K\ge0$. Combining this with \eqref{equ:BA-positive} and \eqref{equ:A-difference-H} proves the lemma.
\end{proof}

Conjugation invariance and optimal eigenvalue matching lift the preceding diagonal estimate to the full matrix space. The resulting auxiliary potential is precisely the input required by the abstract criterion in Section~\ref{sec:criterion}.

\begin{proposition}[Global A-type orbit-distance comparison]\label{prop:gena}
For $Q\in\mathfrak h_0(5)$ define
\begin{equation}\label{equ:def-VA}
V_A(Q):=
\begin{cases}
\displaystyle
V\left(\frac{\delta}{d_A}\Phi_A^-+\left(1-\frac{\delta}{d_A}\right)\Phi_A^+\right), &d(Q,\Sigma_A^+)=\delta<d_A/2,\\[6pt]
\displaystyle
V\left(\frac{\delta}{d_A}\Phi_A^++\left(1-\frac{\delta}{d_A}\right)\Phi_A^-\right),&d(Q,\Sigma_A^-)=\delta<d_A/2,\\[6pt]
0,&d(Q,\Sigma_A^+)\ge d_A/2\text{ and }d(Q,\Sigma_A^-)\ge d_A/2.
\end{cases}
\end{equation}
Then $V(Q)\ge V_A(Q)$ for every $Q\in\mathfrak h_0(5)$.
\end{proposition}

\begin{proof}
The two first cases in \eqref{equ:def-VA} cannot overlap, by the triangle inequality and $d(\Sigma_A^-,\Sigma_A^+)=d_A$. If $Q$ is diagonal, eigenvalue matching shows that its distance to the full unitary orbit $\Sigma_A^+$ equals its distance to $\Sigma_{A,\mathbb R}^+$: both are obtained by matching the two largest eigenvalues of $Q$ with the two entries $3u$ of $\Phi_A^+$. Unitary diagonalization, conjugation invariance, and Lemma~\ref{lem:reala} therefore prove the first case. The second follows by applying the first to $-Q$, since $V(-Q)=V(Q)$ and $\Sigma_A^-=-\Sigma_A^+$.

In the remaining region $V_A=0$, while $V\ge0$ follows from Proposition~\ref{prop:vacuum-characterization}. This proves the proposition without a separate middle-region polynomial estimate.
\end{proof}

We next treat the $B$-type kink. In this regime the comparison can be proved directly after diagonalization. Although the algebraic positivity argument differs from the A-type case, the output has the same geometric form.

\begin{proposition}[Global B-type orbit-distance comparison]\label{prop:genb}
For $Q\in\mathfrak h_0(5)$ define
\begin{equation}\label{equ:def-VB}
    V_B(Q):=
\begin{cases}
\displaystyle
V\left(\frac{\delta}{d_B}\Phi_B^-+
 \left(1-\frac{\delta}{d_B}\right)\Phi_B^+\right),
&d(Q,\Sigma_B^+)=\delta<d_B/2,\\[6pt]
\displaystyle
V\left(\frac{\delta}{d_B}\Phi_B^++
 \left(1-\frac{\delta}{d_B}\right)\Phi_B^-\right),
&d(Q,\Sigma_B^-)=\delta<d_B/2,\\[6pt]
0,&d(Q,\Sigma_B^+)\ge d_B/2
   \text{ and }d(Q,\Sigma_B^-)\ge d_B/2.
\end{cases}
\end{equation}
Then $V(Q)\ge V_B(Q)$ for every $Q\in\mathfrak h_0(5)$.
\end{proposition}

\begin{proof}
As before, the two near-orbit cases are disjoint and unitary invariance reduces the proof to a real diagonal matrix. Write $Q=v\diag(x_1,\ldots,x_5)$ with $x_1\ge\cdots\ge x_5$ and $\sum x_i=0$. Set $\Sigma_{B,\mathbb R}^+:=\{P^T\Phi_B^+P:P\in\SO(5)\}$. Eigenvalue matching gives
\[
    d(Q,\Sigma_B^+)^2=d(Q,\Sigma_{B,\mathbb R}^+)^2=v^2\left(\sum_{i=1}^4(x_i-1)^2+(x_5+4)^2\right).
\]
Indeed, both minima pair the smallest eigenvalue $x_5$ with the entry $-4$ and the remaining four eigenvalues with the repeated entry $1$. The case $d(Q,\Sigma_B^+)=0$ is immediate.  Otherwise let
\[
    \beta:=\frac{d(Q,\Sigma_B^+)}{d_B}\in(0,1/2),\qquad y_i:=x_i-1\ (i\le4),\quad y_5:=x_5+4,\qquad z_i:=\frac{y_i}{\beta}.
\]
Then
\begin{equation}\label{equ:B-z-sphere}
    \sum_{i=1}^5z_i=0,\qquad \sum_{i=1}^5z_i^2=30.
\end{equation}
A direct expansion, with $\Lambda_B^+:=\beta\Phi_B^-+(1-\beta)\Phi_B^+$, gives
\begin{equation*}
    V(Q)-V(\Lambda_B^+)=\frac{9\mu^4}{10h}\,\beta^4\left(A_B(z)+\frac1\beta B_B(z)\right),
\end{equation*}
where
\begin{align*}
    A_B(z):=\frac{21}{10}-\frac1{100}\sum_{i=1}^5z_i^4,\quad B_B(z):=12-\frac{27}{5}z_5+\frac15z_5^3-\frac1{25}\sum_{i=1}^5z_i^3.
\end{align*}
We prove, using exact one-variable certificates, that
\begin{equation}\label{equ:BB-two-claims}
    B_B(z)\ge0,\qquad A_B(z)+2B_B(z)\ge0
\end{equation}
under \eqref{equ:B-z-sphere}. These inequalities imply $A_B+\beta^{-1}B_B\ge A_B+2B_B\ge0$.

Put $t:=z_5$. From the constraints on $z_1,\ldots,z_4$,
\begin{equation}\label{equ:t-range}
    |t|\le2\sqrt6.
\end{equation}
At the two endpoints $|t|=2\sqrt6$, the first four coordinates are all $-t/4$, and the estimates below follow directly (equivalently, by continuity from the interior). We may therefore assume $|t|<2\sqrt6$, where the two fixed-moment constraint gradients are independent. First,
\[
    25B_B=300-135t+4t^3-\sum_{i=1}^4z_i^3.
\]
For fixed $t$, an extremum of $\sum_{i=1}^4z_i^3$ under fixed first and second moments has at most two distinct coordinates. For multiplicities $1+3$, $2+2$, and $3+1$, direct substitution gives the following possible values of $25B_B$:
\begin{align}
    C_1(t)&=\frac5{24}\left[L_0(t)-\sqrt{15}(24-t^2)^{3/2}\right],\label{equ:C1}\\
    C_2(t)&=\frac{25}{8}\left(t^3-36t+96\right),\label{equ:C2}\\
    C_3(t)&=\frac5{24}\left[L_0(t)+\sqrt{15}(24-t^2)^{3/2}\right],\label{equ:C3}
\end{align}
where $L_0(t)=15(t^3-36t+96)$. The minimum of the cubic $t^3-36t+96$ on \eqref{equ:t-range} is $96-48\sqrt3>0$, attained at $t=2\sqrt3$. Moreover,
\begin{equation}\label{equ:L0-square}
    L_0(t)^2-15(24-t^2)^3=240(t-3)^2Q_4(t),
\end{equation}
where
\[
    Q_4(t)=t^4+6t^3-45t^2-144t+864=(24-t^2)^2+3(2t^3+t^2-48t+96).
\]
The last cubic has its minimum $352/27$ on $[-2\sqrt6,2\sqrt6]$. Hence $Q_4>0$; since $L_0>0$, \eqref{equ:L0-square} proves $C_1\ge0$. Equations \eqref{equ:C2}--\eqref{equ:C3} are then nonnegative as well. Thus $B_B\ge0$.

It remains to prove the second inequality in \eqref{equ:BB-two-claims}.
Set
\begin{equation*}
    D(z):=100(A_B(z)+2B_B(z))=2610-\sum_{i=1}^5z_i^4-8\sum_{i=1}^5z_i^3+40t^3-1080t.
\end{equation*}
For fixed $t$, minimizing $D$ is equivalent to maximizing $\sum_{i=1}^4\phi(z_i)$ with $\phi(s)=s^4+8s^3$ under fixed first and second moments. The Lagrange equation is cubic. If three distinct roots $a,b,c$ occur among four coordinates, one of them, say $a$, is repeated. The three roots of the cubic have sum $-6$; the moment constraints then give
\[
    a=6-t,\qquad b+c=t-12,\qquad bc=2t^2-24t+93.
\]
The discriminant of the quadratic with roots $b$ and $c$ is
\[
    \Delta(t):=(b-c)^2=(b+c)^2-4bc=-7t^2+72t-228.
\]
Since $\Delta'(t)=72-14t>0$ on \eqref{equ:t-range},
\[
    \Delta(t)\le \Delta(2\sqrt6)=144\sqrt6-396<0.
\]
Thus $b$ and $c$ cannot both be real, and this case is impossible. Hence only the multiplicities $1+3$, $2+2$, and $3+1$ occur. The corresponding values of $D$ are
\begin{align}
    D_1(t)&=\frac5{24}\left[L_1(t)-\sqrt{15}(8-t)(24-t^2)^{3/2}\right],\label{equ:D1}\\
    D_2(t)&=\frac5{16}R_2(t),\label{equ:D2}\\
    D_3(t)&=\frac5{24}\left[L_1(t)+\sqrt{15}(8-t)(24-t^2)^{3/2}\right],\label{equ:D3}
\end{align}
where
\begin{align*}
    L_1(t)&=-7t^4+120t^3+156t^2-4320t+10008,\\
    R_2(t)&=-3t^4+80t^3+24t^2-2880t+7632.
\end{align*}
The exact Bernstein certificates in Appendix~\ref{app:certificates} show that $L_1$, $R_2$, and
\[
    Q_6(t):=t^6-24t^5+36t^4+2232t^3-5616t^2-43794t+150849
\]
are positive on $[-5,5]$.  Finally,
\begin{equation}\label{equ:L1-square}
    L_1(t)^2-15(8-t)^2(24-t^2)^3=64(t-3)^2Q_6(t).
\end{equation}
Since $8-t>0$ on \eqref{equ:t-range}, equations \eqref{equ:D1} and \eqref{equ:L1-square} give $D_1\ge0$. Also $D_3\ge D_1$ and $D_2>0$. Therefore $D\ge0$, proving \eqref{equ:BB-two-claims} and hence the first branch of \eqref{equ:def-VB}. The second branch follows from $Q\mapsto-Q$.

In the remaining region $V_B=0$, while $V\ge0$ follows from Proposition~\ref{prop:vacuum-characterization}. This completes the proof.
\end{proof}

\section{The distance--coarea lower bound and rigidity}
\label{sec:criterion}
We isolate the part of the argument that does not depend on the detailed form of the $\SU(5)$ potential. This is the model-independent variational core of the paper: it converts a near-vacuum lower bound expressed through distance into a sharp global action estimate. Let $\mathcal H$ be a finite-dimensional real Hilbert space, and let $S^-$ and $S^+$ be disjoint nonempty compact subsets of $\mathcal H$. Set
\[
    D:=d(S^-,S^+)>0,\qquad\rho(q):=\min\{d(q,S^-),d(q,S^+)\}.
\]

The purpose of the following criterion is to reduce the Hilbert-space-valued connection problem to the variation of the single scalar function $\rho\circ\Psi$. The calibration inequality controls the energy by the weighted total variation of this distance function, while the one-dimensional coarea formula records how many times the connection crosses each distance level. Since every path joining $S^-$ to $S^+$ must move from distance $0$ to at least $D/2$ and return to distance $0$, each intermediate level is crossed at least twice. This elementary topological fact is the source of the factor $4$ in the sharp lower bound.

\begin{proposition}[Distance--coarea criterion]\label{prop:criterion}
Let $W:\mathcal H\to[0,\infty)$ be continuous, and let $g:[0,D/2)\to[0,\infty)$ be Borel measurable and locally bounded. Suppose that
\begin{equation}\label{equ:local-distance-lower-bound}
    W(q)\ge g(\rho(q))\qquad\text{whenever }\rho(q)<D/2.
\end{equation}
Then every $\Psi\in H^1_{\mathrm{loc}}(\mathbb R;\mathcal H)$ with finite energy and limits $\Psi(-\infty)\in S^-$ and $\Psi(+\infty)\in S^+$ satisfies
\begin{equation}\label{equ:general-lower-bound}
    \int_{\mathbb R}\bigl(\lVert\Psi'\rVert^2+W(\Psi)\bigr)\,\mathrm dz\ge 4\int_0^{D/2}\sqrt{g(s)}\,\mathrm ds.
\end{equation}
Equality holds if, in addition,
\[
    W(\Psi)=g(\rho(\Psi))\quad\text{for a.e. }z\text{ with }\rho(\Psi)<D/2,
\]
\[
    \lVert\Psi'\rVert^2=W(\Psi),\qquad|(\rho\circ\Psi)'|=\lVert\Psi'\rVert\quad\text{a.e.},
\]
and $\rho\circ\Psi$ increases from $0$ to $D/2$ and then decreases to $0$, with its $D/2$-level set of measure zero.
\end{proposition}

\begin{proof}
Extend $g$ to a nonnegative Borel function $\widetilde g$ on $[0,\infty)$ by setting $\widetilde g(s)=g(s)$ for $s<D/2$ and $\widetilde g(s)=0$ for $s\ge D/2$. Then $W(q)\ge\widetilde g(\rho(q))$ for every $q\in\mathcal H$. The distance to a nonempty closed set is $1$-Lipschitz, hence $\rho\circ\Psi$ is locally absolutely continuous and
\[
    |(\rho\circ\Psi)'|\le\lVert\Psi'\rVert\quad\text{a.e.}
\]
The elementary inequality $a^2+b^2\ge2ab$ gives
\begin{equation}\label{equ:calibration}
    \int_{\mathbb R}\bigl(\lVert\Psi'\rVert^2+W(\Psi)\bigr)\,\mathrm dz\ge 2\int_{\mathbb R}\sqrt{W(\Psi)}\,\lVert\Psi'\rVert\,\mathrm dz\nonumber\ge 2\int_{\mathbb R}\sqrt{\widetilde g(\rho(\Psi))}\,|(\rho\circ\Psi)'|\,\mathrm dz.
\end{equation}
The continuous function $d(\Psi(z),S^-)-d(\Psi(z),S^+)$ changes sign between the two ends. At a zero $z_0$, the triangle inequality gives $d(\Psi(z_0),S^\pm)\ge D/2$, hence $\rho(\Psi(z_0))\ge D/2$. Set $r:=\rho\circ\Psi$, and let
\[
N_r(s):=\#\{z\in\mathbb R:r(z)=s\}
\]
denote the multiplicity of the level $s$, with the usual convention that
$N_r(s)=+\infty$ when the level set is infinite.  Applying the one-dimensional area formula to $r|_{[-R,R]}$ and then
letting $R\to\infty$ by monotone convergence gives
\cite[Section~3.4]{EvansGariepy2015}
\begin{equation}\label{equ:coarea-multiplicity}
\int_{\mathbb R}
\sqrt{\widetilde g(r(z))}\,|r'(z)|\,\mathrm dz
=
\int_0^\infty \sqrt{\widetilde g(s)}\,N_r(s)\,\mathrm ds.
\end{equation}
Since $r(z)\to0$ as $z\to\pm\infty$ and $r(z_0)\ge D/2$, continuity implies
\[
    N_r(s)\ge2\qquad\text{for a.e. }s\in(0,D/2).
\]
Consequently,
\begin{align*}
    2\int_{\mathbb R}\sqrt{\widetilde g(r(z))}\,|r'(z)|\,\mathrm dz=2\int_0^\infty\sqrt{\widetilde g(s)}\,N_r(s)\,\mathrm ds\ge4\int_0^{D/2}\sqrt{g(s)}\,\mathrm ds.
\end{align*}
This proves \eqref{equ:general-lower-bound}. The stated a.e.\ conditions make every inequality an equality.
\end{proof}

\begin{remark}
The coarea formula is used here not merely to reparametrize the action, but to retain the multiplicity of the distance levels crossed by a connection. The estimate $N_r(s)\ge2$ expresses the unavoidable passage from one vacuum component to the other, while equality $N_r(s)=2$ for almost every intermediate level is part of the rigidity mechanism for a sharp minimizer. Thus the same scalar reduction yields both the global lower bound and the geometric information needed for the attainment criterion.
\end{remark}

For orbit-to-orbit minimization, the lower bound is sufficient. The fixed-endpoint problem additionally requires an equality analysis: a profile can attain the orbit-to-orbit cost only when its prescribed endpoints are geometrically compatible with the shortest crossing between the two zero sets.

\begin{proposition}[Rigidity at the sharp lower bound]\label{prop:sharp-rigidity}
In addition to the hypotheses of Proposition~\ref{prop:criterion}, suppose that
\[
    W^{-1}(0)=S^-\cup S^+,\qquad g(s)>0\quad\text{for }0<s<D/2.
\]
Let $A^-\in S^-$ and $A^+\in S^+$, and let $\Psi\in H^1_{\mathrm{loc}}(\mathbb R;\mathcal H)$ have finite energy, limits $\Psi(-\infty)=A^-$ and $\Psi(+\infty)=A^+$, and attain the lower bound
\begin{equation}\label{equ:sharp-rigidity-energy}
    \int_{\mathbb R}\bigl(\lVert\Psi'\rVert^2+W(\Psi)\bigr)\,\mathrm dz=4\int_0^{D/2}\sqrt{g(s)}\,\mathrm ds.
\end{equation}
Then
\begin{equation}\label{equ:rigidity-closest-endpoints}
    \lVert A^+-A^-\rVert=D.
\end{equation}
\end{proposition}

\begin{proof}
Let $r:=\rho\circ\Psi$, and extend $g$ by zero on $[D/2,\infty)$ as in the proof of Proposition~\ref{prop:criterion}. Equality in \eqref{equ:sharp-rigidity-energy}, together with the chain of inequalities in the proof of Proposition~\ref{prop:criterion}, forces equality at every step. In particular,
\begin{equation}\label{equ:rigidity-equipartition}
    \lVert\Psi'\rVert^2=W(\Psi)\quad\text{a.e.}
\end{equation}
and, on the set $0<r<D/2$,
\begin{equation}\label{equ:rigidity-radial-speed}
    W(\Psi)=g(r),\qquad |r'|=\lVert\Psi'\rVert\quad\text{a.e.}
\end{equation}
Indeed, the two pointwise gaps
$(\sqrt{W}-\sqrt{g(r)})\lVert\Psi'\rVert$ and
$\sqrt{g(r)}(\lVert\Psi'\rVert-|r'|)$ are nonnegative, and their
integrals must vanish.  On $\{0<r<D/2\}$ one has $g(r)>0$. Moreover,
$\lVert\Psi'\rVert$ cannot vanish on a subset of positive measure there:
by \eqref{equ:rigidity-equipartition}, this would imply $W(\Psi)=0$ and
hence $\Psi\in S^-\cup S^+$, which would give $r=0$. Therefore
$\lVert\Psi'\rVert>0$ a.e.\ on $\{0<r<D/2\}$, and the vanishing of the two
gaps proves \eqref{equ:rigidity-radial-speed}.

On $\{r\ge D/2\}$, equality in the second inequality of the
calibration chain, together with \eqref{equ:rigidity-equipartition},
implies $W(\Psi)=\lVert\Psi'\rVert=0$ a.e. Since $W^{-1}(0)=S^-\cup S^+$ and
$\rho=0$ on this zero set, $\{r\ge D/2\}$ has measure zero. In
particular, the open set $\{r>D/2\}$ is empty.

Let $N_r(s)$ denote the number of preimages of a level $s$. Equality in the coarea step and the strict positivity of $g$ give
\begin{equation}\label{equ:rigidity-level-count}
    N_r(s)=2\quad\text{for a.e. }s\in(0,D/2).
\end{equation}
The function $d(\Psi(z),S^-)-d(\Psi(z),S^+)$ changes sign, so there is $z_0\in\mathbb R$ at which the two distances agree. The triangle inequality gives $r(z_0)\ge D/2$, while the preceding paragraph gives $r\le D/2$ everywhere; hence $r(z_0)=D/2$.

For almost every $s\in(0,D/2)$, continuity and the limits at infinity give at least one preimage of $s$ on each side of $z_0$. By \eqref{equ:rigidity-level-count}, there is exactly one on each side. The one-dimensional coarea formula on the two half-lines therefore yields
\begin{equation}\label{equ:rigidity-half-variation}
    \int_{-\infty}^{z_0}|r'|\,\mathrm dz=\int_{z_0}^{\infty}|r'|\,\mathrm dz=\frac D2.
\end{equation}
Equations \eqref{equ:rigidity-equipartition}--\eqref{equ:rigidity-radial-speed} show that the same identities hold with $\lVert\Psi'\rVert$ in place of $|r'|$: on $\{r=0\}$ one has $W(\Psi)=0$ and hence $\Psi'=0$ a.e., while $\{r\ge D/2\}$ has measure zero. Consequently,
\[
    \frac D2=d(\Psi(z_0),S^-)\le\lVert\Psi(z_0)-A^-\rVert\le\int_{-\infty}^{z_0}\lVert\Psi'\rVert\,\mathrm dz=\frac D2,
\]
and similarly $\lVert A^+-\Psi(z_0)\rVert=D/2$.  It follows that
\[
    \lVert A^+-A^-\rVert\le\lVert A^+-\Psi(z_0)\rVert+\lVert\Psi(z_0)-A^-\rVert=D.
\]
The reverse inequality follows from $A^\pm\in S^\pm$ and the definition of $D$, proving \eqref{equ:rigidity-closest-endpoints}.
\end{proof}

\begin{remark}
Only the near-vacuum comparison \eqref{equ:local-distance-lower-bound} is needed. No positive lower bound is required in the middle region $\rho\ge D/2$; global nonnegativity of $W$ is sufficient.
\end{remark}

\section{Application to the kink \texorpdfstring{$\Phi_A$}{Phi A}}\label{sec:app-a}

We now combine the A-type orbit-distance comparison with the abstract criterion. The lower bound follows immediately; the substantive equality check is that $\Phi_A$ follows a minimizing segment at the equipartition speed.

For $0\le s<d_A/2$, define
\begin{equation*}
    g_A(s):=\frac{8\mu^4}{\lambda}\frac{s^2}{d_A^2}\left(1-\frac{s}{d_A}\right)^2.
\end{equation*}
Let
\[
    \rho_A(Q):=\min\{d_F(Q,\Sigma_A^-),d_F(Q,\Sigma_A^+)\}.
\]
Proposition~\ref{prop:gena} gives
\begin{equation}\label{equ:aux-a}
    V(Q)\ge g_A(\rho_A(Q))\qquad\text{whenever }\rho_A(Q)<d_A/2.
\end{equation}
Proposition~\ref{prop:criterion} therefore yields
\begin{equation}\label{equ:bound-a}
    E(\Psi)\ge4\int_0^{d_A/2}\sqrt{g_A(s)}\,\mathrm ds\quad\text{for every }\Psi\in\mathcal A(\Sigma_A^-,\Sigma_A^+).
\end{equation}

It remains to verify that the explicit profile attains this bound. Put $t(z):=\tanh(\mu z/\sqrt2)$.  Direct calculation gives
\begin{align}
    \lVert\Phi_A'(z)\rVert_F^2=\frac{\mu^4}{2\lambda}\sech^4\left(\frac{\mu z}{\sqrt2}\right),\label{equ:kinetic-a}\quad \Tr(\Phi_A^2)=\frac{5\mu^2}{4\lambda}\left(4+\frac45t^2\right),\quad  \Tr(\Phi_A^4)=\frac{25\mu^4}{16\lambda^2}\left(4+\frac{24}{25}t^2+\frac{52}{125}t^4\right).\nonumber
\end{align}
Using \eqref{equ:parameters-a} in \eqref{equ:potential} yields
\begin{equation}\label{equ:equipartition-a}
    V(\Phi_A(z))=\frac{\mu^4}{2\lambda}(1-t(z)^2)^2=\lVert\Phi_A'(z)\rVert_F^2.
\end{equation}
Moreover,
\[
    \Phi_A(z)=\frac{1-t(z)}{2}\Phi_A^-+\frac{1+t(z)}{2}\Phi_A^+.
\]
The endpoint pair realizes the orbit distance, so Lemma~\ref{lem:minimizing-segment} gives that $\rho_A\circ\Phi_A$ increases from $0$ to $d_A/2$ and then decreases to $0$, with $|(\rho_A\circ\Phi_A)'|=\lVert\Phi_A'\rVert_F$ away from $z=0$. All equality conditions in Proposition~\ref{prop:criterion} therefore hold, and
\[
    E(\Phi_A)=4\int_0^{d_A/2}\sqrt{g_A(s)}\,\mathrm ds=\frac{4\sqrt2\,\mu^3}{3\lambda}.
\]
This proves the energy assertion in Theorem~\ref{thm:globala}. For any absolutely continuous curve $\gamma$ joining the two vacuum orbits, the same Lipschitz and coarea argument, without the kinetic term, gives
\[
    2\int_0^1\sqrt{V(\gamma)}\,\lVert\gamma'\rVert_F\,\mathrm dt\ge4\int_0^{d_A/2}\sqrt{g_A(s)}\,\mathrm ds.
\]
The minimizing segment traced by $\Phi_A$ attains equality. Hence $\mathsf d_V(\Sigma_A^-,\Sigma_A^+)=E(\Phi_A)$.

\section{Application to the kink \texorpdfstring{$\Phi_B$}{Phi B}}\label{sec:app-b}

The B-type application has the same variational structure.  The different coefficient in the auxiliary potential changes the wall tension but not the mechanism by which the explicit profile attains the sharp bound.

For $0\le s<d_B/2$, define
\begin{equation}\label{equ:g-b}
    g_B(s):=\frac{27\mu^4}{5h}\frac{s^2}{d_B^2}\left(1-\frac{s}{d_B}\right)^2.
\end{equation}
Let
\[
    \rho_B(Q):=\min\{d_F(Q,\Sigma_B^-),d_F(Q,\Sigma_B^+)\}.
\]
Proposition~\ref{prop:genb} gives
\begin{equation}\label{equ:aux-b}
    V(Q)\ge g_B(\rho_B(Q))\qquad\text{whenever }\rho_B(Q)<d_B/2.
\end{equation}
Proposition~\ref{prop:criterion} therefore yields
\begin{equation}\label{equ:bound-b}
    E(\Psi)\ge4\int_0^{d_B/2}\sqrt{g_B(s)}\,\mathrm ds\quad\text{for every }\Psi\in\mathcal A(\Sigma_B^-,\Sigma_B^+).
\end{equation}

With $t(z):=\tanh(\mu z/\sqrt2)$, direct substitution gives
\begin{equation}\label{equ:equipartition-b}
    \lVert\Phi_B'(z)\rVert_F^2=V(\Phi_B(z))=\frac{27\mu^4}{80h}\sech^4\left(\frac{\mu z}{\sqrt2}\right).
\end{equation}
The profile has the line-segment representation
\[
    \Phi_B(z)=\frac{1-t(z)}{2}\Phi_B^-+\frac{1+t(z)}{2}\Phi_B^+.
\]
Lemma~\ref{lem:minimizing-segment} again verifies all equality conditions, so
\[
    E(\Phi_B)=4\int_0^{d_B/2}\sqrt{g_B(s)}\,\mathrm ds=\frac{9\sqrt2\,\mu^3}{10h}.
\]
This proves the energy assertion in Theorem~\ref{thm:globalb}. Applying the corresponding weighted coarea estimate to arbitrary curves and using the minimizing segment traced by $\Phi_B$ gives $\mathsf d_V(\Sigma_B^-,\Sigma_B^+)=E(\Phi_B)$.

\section{Fixed-endpoint sectors and non-attainment}
\label{sec:fixed-endpoints}

We now prove Theorem~\ref{thm:fixed-endpoint-sectors}.  Two ingredients are needed beyond orbit-to-orbit minimality. First, the closest endpoint pairs must be characterized explicitly. Second, arbitrary prescribed endpoints must be connected to such a pair by increasingly slow motion along the zero-energy vacuum orbits.

\begin{lemma}[Closest pairs of vacuum representatives]\label{lem:closest-pairs}
For $A^+\in\Sigma_A^+$ and $A^-\in\Sigma_A^-$ there are rank-two orthogonal projections $P,Q$ such that
\[
    A^+=u(-2I+5P),\qquad A^-=u(2I-5Q),
\]
and
\begin{equation}\label{equ:A-projection-distance}
    \lVert A^+-A^-\rVert_F^2=20u^2+50u^2\Tr(PQ).
\end{equation}
For $A^+\in\Sigma_B^+$ and $A^-\in\Sigma_B^-$ there are rank-one orthogonal projections $P,Q$ such that
\[
    A^+=v(I-5P),\qquad A^-=v(-I+5Q),
\]
and
\begin{equation}\label{equ:B-projection-distance}
    \lVert A^+-A^-\rVert_F^2=30v^2+50v^2\Tr(PQ).
\end{equation}
In either case the endpoint pair realizes the orbit distance if and only if $PQ=0$. Every such closest pair is a simultaneous $\SU(5)$ conjugate of the corresponding standard pair $(\Phi_\ast^-,\Phi_\ast^+)$.
\end{lemma}

\begin{proof}
The projection representations follow directly from the two eigenvalue multiplicities. Expanding the Frobenius norms and using $P^2=P$, $Q^2=Q$ gives \eqref{equ:A-projection-distance} and \eqref{equ:B-projection-distance}. Moreover,
\[
    \Tr(PQ)=\Tr(PQP)=\lVert QP\rVert_F^2\ge0,
\]
with equality if and only if the ranges of $P$ and $Q$ are orthogonal. Equations \eqref{equ:orbit-distances}, \eqref{equ:A-projection-distance}, and \eqref{equ:B-projection-distance} therefore give the closest-pair criterion. An ordered pair of orthogonal subspaces of the relevant ranks can be mapped to the standard coordinate pair by a unitary transformation; a scalar phase adjusts its determinant without changing the conjugation, so the transformation may be chosen in $\SU(5)$.
\end{proof}

\begin{proof}[Proof of Theorem~\ref{thm:fixed-endpoint-sectors}]
Fix $\ast\in\{A,B\}$ and $A^\pm\in\Sigma_\ast^\pm$.  Since $\mathcal A(A^-,A^+)\subset\mathcal A(\Sigma_\ast^-,\Sigma_\ast^+)$, Theorems~\ref{thm:globala}--\ref{thm:globalb} give
\begin{equation}\label{equ:fixed-lower-bound}
    \inf_{\Psi\in\mathcal A(A^-,A^+)}E(\Psi)\ge\sigma_\ast.
\end{equation}

To prove the reverse inequality, choose a closest pair $P^-\in\Sigma_\ast^-$, $P^+\in\Sigma_\ast^+$ and, by Lemma~\ref{lem:closest-pairs}, a conjugate $\widehat\Phi_\ast$ of the explicit kink joining $P^-$ to $P^+$.  Because the vacuum orbits are connected smooth homogeneous manifolds, there are piecewise $C^1$ curves
\[
    \gamma^-:[0,1]\to\Sigma_\ast^-,\qquad\gamma^-(0)=A^-,\quad\gamma^-(1)=P^-,
\]
and
\[
    \gamma^+:[0,1]\to\Sigma_\ast^+,\qquad\gamma^+(0)=P^+,\quad\gamma^+(1)=A^+.
\]
For $T>0$, define $\gamma_T^\pm(z):=\gamma^\pm(z/T)$ on $[0,T]$. Traversing the orbit curves in this way costs only
\begin{equation}\label{equ:slow-orbit-cost}
    \int_0^T\left(\lVert(\gamma_T^\pm)'\rVert_F^2+V(\gamma_T^\pm)\right)\,\mathrm dz=\frac1T\int_0^1\lVert(\gamma^\pm)'(s)\rVert_F^2\,\mathrm ds,
\end{equation}
because $V=0$ on the vacuum orbits.

For $L>0$, set $Q_L^\pm:=\widehat\Phi_\ast(\pm L)$ and connect $P^-$ to $Q_L^-$ and $Q_L^+$ to $P^+$ by the unit-time affine segments
\[
    \eta_L^-(s)=(1-s)P^-+sQ_L^-, \qquad\eta_L^+(s)=(1-s)Q_L^++sP^+.
\]
Since $Q_L^\pm\to P^\pm$, and $P^\pm$ are global minima of the smooth
potential $V$, one has $V(P^\pm)=0$ and $DV(P^\pm)=0$.  Hence there are
$C>0$ and $L_0>0$ such that, for $L\ge L_0$ and $0\le s\le1$,
\[
0\le V\bigl(P^-+s(Q_L^--P^-)\bigr)
   \le C\lVert Q_L^--P^-\rVert_F^2
\]
and the analogous estimate holds for the $+$ connector.  Consequently,
\begin{equation}\label{equ:connector-cost}
    e_L^-+e_L^+\longrightarrow0\qquad\text{as }L\to\infty,
\end{equation}
where
\[
    e_L^\pm:=\int_0^1\left(\lVert(\eta_L^\pm)'(s)\rVert_F^2
    +V(\eta_L^\pm(s))\right)\,\mathrm ds,
\]
and, more explicitly,
\[
    e_L^\pm\le (1+C)\lVert Q_L^\pm-P^\pm\rVert_F^2
\]
for all sufficiently large $L$.
Translate the five pieces $\gamma_T^-$, $\eta_L^-$, $\widehat\Phi_\ast|_{[-L,L]}$, $\eta_L^+$, and $\gamma_T^+$ to adjacent intervals, concatenate them, and extend the resulting path constantly by $A^-$ and $A^+$ at the two ends. This gives $\Psi_{L,T}\in\mathcal A(A^-,A^+)$ and
\begin{align}\label{equ:recovery-energy}
    E(\Psi_{L,T})=\frac1T\int_0^1\lVert(\gamma^-)'\rVert_F^2\,\mathrm ds+e_L^-+\int_{-L}^{L}\bigl(\lVert\widehat\Phi_\ast'\rVert_F^2+V(\widehat\Phi_\ast)\bigr)\,\mathrm dz\notag+e_L^++\frac1T\int_0^1\lVert(\gamma^+)'\rVert_F^2\,\mathrm ds.
\end{align}
Letting first $L\to\infty$ and then $T\to\infty$ gives $\limsup E(\Psi_{L,T})\le\sigma_\ast$. Together with \eqref{equ:fixed-lower-bound}, this proves
\eqref{equ:fixed-endpoint-infimum}.

If $\lVert A^+-A^-\rVert_F=d_\ast$, Lemma~\ref{lem:closest-pairs} gives $U\in\SU(5)$ such that $A^\pm=U^\dagger\Phi_\ast^\pm U$. Hence $U^\dagger\Phi_\ast(z-z_0)U$ belongs to $\mathcal A(A^-,A^+)$ and has energy $\sigma_\ast$, so the infimum is attained. Conversely, if a minimizer exists, its energy equals $\sigma_\ast$, which is exactly the sharp lower bound in Proposition~\ref{prop:criterion}.  Proposition~\ref{prop:sharp-rigidity}, using Proposition~\ref{prop:vacuum-characterization} and the strict positivity of $g_A,g_B$ on $(0,d_\ast/2)$, then implies $\lVert A^+-A^-\rVert_F=d_\ast$. This proves the attainment criterion and non-attainment for every nonclosest endpoint pair.
\end{proof}

\begin{remark}\label{rmk:minimizer-symmetry}
For every $U\in\SU(5)$ and $z_0\in\mathbb R$,
\[
    \Phi_{\ast,U,z_0}(z):=U^\dagger\Phi_\ast(z-z_0)U,\qquad \ast\in\{A,B\},
\]
is a minimizer in the corresponding orbit-to-orbit class. In a fixed-endpoint class it is admissible precisely when the conjugation sends the standard endpoint pair to the prescribed pair. Translation reflects the autonomy of the energy, while conjugation reflects its $\SU(5)$ symmetry. We do not claim that these are all orbit-to-orbit minimizers; a complete classification would require a systematic analysis of all equality cases in the orbit-distance polynomial estimates.
\end{remark}

\begin{remark}[Loss of compactness]\label{rmk:loss-compactness}
For a nonclosest endpoint pair, the recovery sequence above spreads the motion along each vacuum orbit over intervals whose lengths tend to infinity. The tangential kinetic cost then tends to zero by \eqref{equ:slow-orbit-cost}, while the central kink retains the wall tension. This explicitly identifies the mechanism responsible for non-attainment in that fixed-endpoint sector.
\end{remark}

\section{Conclusion}\label{sec:conclusion}
We have developed an orbit-to-orbit variational framework for the two distinguished $\SU(5)\times\mathbb Z_2$ parameter regimes. The full vacuum sets are pairs of compact homogeneous conjugacy orbits, so the corresponding wall problems are manifold-to-manifold rather than point-to-point. Within this framework, the explicit non-Abelian kinks minimize the energy among all finite-energy matrix-valued connections joining the two vacuum components, without restrictions to diagonal, commuting, or symmetry-reduced competitors.  Their exact wall tensions are
\[
    \sigma_A=\frac{4\sqrt2\,\mu^3}{3\lambda},\qquad\sigma_B=\frac{9\sqrt2\,\mu^3}{10h}.
\]
Equivalently, the two minimizing line segments realize the degenerate weighted distances between the corresponding vacuum manifolds.

The proof separates two mechanisms. The model-independent part is the distance--coarea criterion, which converts a lower bound depending only on the distance to two compact zero sets into a sharp action bound. The model-specific part is the construction of the auxiliary potentials: it requires optimal spectral matching between conjugacy orbits and exact polynomial positivity on constrained trace spheres. This separation makes clear which part of the method can be reused for other matrix-valued multiwell potentials and which part must be re-established for each new model.

The fixed-endpoint analysis further exposes the role of the continuous vacuum geometry. Every prescribed endpoint pair on the two vacuum components has the same energy infimum $\sigma_\ast$. The infimum is attained exactly for closest pairs; for nonclosest pairs it is approached by profiles that move increasingly slowly along the zero-energy orbits and is not attained. Thus the vacuum directions do not change the optimal tension, but they do create a precise loss-of-compactness mechanism absent from isolated-well problems.

The conclusions are limited to the special coupling ratios in \eqref{equ:parameters-a} and \eqref{equ:parameters-b} and to one-dimensional static walls.  General coupling ratios, different pairs of vacuum orbits, and larger $\SU(N)$ models require new sharp orbit-distance comparisons. Other natural questions include a complete classification of all equality cases and orbit-to-orbit minimizers, and the role of the tensions obtained here in higher-dimensional sharp-interface limits.

\begin{appendices}

\section{Stationary-point reductions for the A-type comparison}
\label{app:a-stationary}

For completeness, we record the exact eliminations used in the proof of
Lemma~\ref{lem:reala}.  This also makes explicit why the stationary-point
lists in the two positivity arguments are exhaustive.

\subsection{The polynomial \texorpdfstring{$B_A$}{B A}}
At a constrained stationary point, equations
\eqref{equ:BA-stationary} hold together with
\[
\sum_{k=1}^5z_k=0,\qquad \sum_{k=1}^5z_k^2=20.
\]
If $z_3=z_4=z_5=\theta$, the two constraints give, with
$s=z_1+z_2$ and $p=z_1z_2$,
\[
s=-3\theta,\qquad p=6\theta^2-10.
\]
For $z_1=z_2$ this gives the two points listed in Step~1.  If
$z_1\ne z_2$, subtracting the last two equations in
\eqref{equ:BA-stationary} and eliminating the multipliers reduces the
remaining condition to
\[
\theta^2-\theta-1=0.
\]
The quadratic with roots $z_1,z_2$ is then
$X^2+3\theta X+6\theta^2-10$, and substitution of the preceding relation
shows that $-2$ is one of its roots.

Now suppose $z_3=z_4=a$, $z_5=b$ with $a\ne b$.  If $z_1=z_2=r$,
exact elimination of $b,r,\sigma_1,\sigma_2$ gives
\[
(a-1)(a+1)(a^2-5)(3a^2-8)=0.
\]
The factor $3a^2-8$ yields $a=b=\pm2\sqrt6/3$ and hence belongs to the
preceding all-equal case.  For the remaining roots, substitution gives
\[
(a,b,r)=(-1,4,-1),\ (1,-4,1),\
(\sqrt5,0,-\sqrt5),\ (-\sqrt5,0,\sqrt5).
\]

Finally assume $z_1\ne z_2$.  Subtracting the equations corresponding to
$a$ and $b$, and separately those corresponding to $z_1$ and $z_2$, gives
\[
\sigma_2=12(a+b),\qquad
\sigma_2=-18s-60,
\qquad s:=z_1+z_2.
\]
Together with $s+2a+b=0$, these identities yield
\[
b=10-4a,\qquad s=2a-10.
\]
The sum of the $z_1$ and $z_2$ equations gives
\[
p:=z_1z_2=\frac{8a^2-30a+45}{3},
\]
and the second-moment constraint then reduces to
\[
\frac{50}{3}(a-3)^2=0.
\]
Thus $a=3$, $b=-2$, $s=-4$, and $p=9$, for which
$s^2-4p=-20$.  Hence no real stationary point occurs in this final case.

\subsection{The polynomial \texorpdfstring{$\mathcal K$}{K}}
The constrained stationary equations are precisely
\eqref{equ:K-stationary}, together with
\[
\sum_{k=1}^5w_k=0,\qquad \sum_{k=1}^5w_k^2=5.
\]
If $w_3=w_4=w_5=\theta$, set
$s=w_1+w_2$ and $p=w_1w_2$.  The constraints give
\[
s=-3\theta,\qquad p=6\theta^2-\frac52.
\]
For $w_1=w_2$ this yields the two points listed in Step~2.  If
$w_1\ne w_2$, subtracting the last two equations in
\eqref{equ:K-stationary} and eliminating the multipliers gives
\[
16\theta^3-60\theta^2+26\theta+15=0,
\]
while the requirement that $w_1,w_2$ be real is
\[
(w_1-w_2)^2=s^2-4p=10-15\theta^2\ge0.
\]

Suppose next that $w_3=w_4=a$, $w_5=b$ with $a\ne b$ and
$w_1=w_2=r$.  Eliminating $b,r,\sigma_1,\sigma_2$ gives
\[
(2a-1)(2a+1)(3a^2-2)(4a^2-5)(4a^2+39)=0.
\]
The factor $3a^2-2$ again gives $a=b$ and belongs to the all-equal case,
whereas $4a^2+39$ has no real root.  The remaining roots give exactly
\[
(a,b,r)=\left(-\frac12,2,-\frac12\right),
\left(\frac12,-2,\frac12\right),
\left(\frac{\sqrt5}{2},0,-\frac{\sqrt5}{2}\right),
\left(-\frac{\sqrt5}{2},0,\frac{\sqrt5}{2}\right).
\]

If $w_1\ne w_2$, subtraction of the relevant stationary equations,
followed by the two constraints, gives
\[
2ab-2b^2-24a-6b+35=0.
\]
Since $|b|\le\sqrt5<12$, this determines
$a=(2b^2+6b-35)/(2b-24)$.  The discriminant of the quadratic with roots
$w_1,w_2$ is then the negative expression displayed in the proof of
Lemma~\ref{lem:reala}, so this case has no real solution.  If
$w_3,w_4,w_5$ are all distinct, they are all three roots of the first
cubic in \eqref{equ:K-stationary}; their sum is therefore $6$, contradicting
$w_3^2+w_4^2+w_5^2\le5$.  This completes the stationary classification.

\section{Exact Bernstein certificates}\label{app:certificates}

For completeness, we give the exact rational certificates used in the proof of Proposition~\ref{prop:genb}. If $P$ has degree $n$, write
\[
    P(a+(b-a)u)=\sum_{j=0}^n c_j\binom{n}{j}u^j(1-u)^{n-j},\qquad 0\le u\le1,
\]
and denote $(c_0,\ldots,c_n)$ by $\mathcal B_{[a,b]}(P)$. Positivity of all entries proves $P>0$ on $[a,b]$.

For
\[
    L_1(t)=-7t^4+120t^3+156t^2-4320t+10008
\]
the coefficient vectors are
\begin{align*}
    \mathcal B_{[-5,-4]}(L_1)&=(16133,17788,18994,19816,20312),\\
    \mathcal B_{[-4,0]}(L_1)&=(20312,22296,19064,14328,10008),\\
    \mathcal B_{[0,4]}(L_1)&=(10008,5688,1784,216,1112),\\
    \mathcal B_{[4,5]}(L_1)&=(1112,1336,1714,2248,2933).
\end{align*}
For
\[
    R_2(t)=-3t^4+80t^3+24t^2-2880t+7632
\]
we have
\begin{align*}
    \mathcal B_{[-5,-4]}(R_2)&=(10757,11852,12676,13264,13648),\\
    \mathcal B_{[-4,0]}(R_2)&=(13648,15184,13456,10512,7632),\\
    \mathcal B_{[0,4]}(R_2)&=(7632,4752,1936,464,848),\\
    \mathcal B_{[4,5]}(R_2)&=(848,944,1156,1492,1957).
\end{align*}
Finally, for
\[
    Q_6(t)=t^6-24t^5+36t^4+2232t^3-5616t^2-43794t+150849
\]
we obtain
\begin{align*}
    \mathcal B_{[-5,-4]}(Q_6)&=(63544,74880,432973/5,491692/5,549147/5,120844,131209),\\
    \mathcal B_{[-4,4]}(Q_6)&=(131209,214129,1277373/5,832837/5,325837/5,13329,17401),\\
    \mathcal B_{[4,5]}(Q_6)&=(17401,17910,96463/5,107806/5,123573/5,28738,33604).
\end{align*}
Every listed coefficient is strictly positive. These intervals cover $[-5,5]$, which contains $[-2\sqrt6,2\sqrt6]$. The certificate uses only exact integer and rational arithmetic.
\end{appendices}

\section*{Acknowledgements}
The authors would like to thank Professor Wei Wang for his guidance and
valuable comments and suggestions on the manuscript.

\section*{Statements and Declarations}

\noindent\textbf{Funding.}
No funding was received for conducting this study.

\medskip
\noindent\textbf{Author Contributions.}
Both authors contributed to the conception and mathematical analysis
of the study. Sisi Guan prepared the initial draft of the manuscript.
Both authors reviewed and approved the final manuscript.

\medskip
\noindent\textbf{Data Availability.}
No datasets were generated or analysed during the current study.

\medskip
\noindent\textbf{Competing Interests.}
The authors have no relevant financial or non-financial interests to disclose.

\bibliography{sn-bibliography}
\end{document}